\documentclass{amsart}

\usepackage{amssymb}

\usepackage{graphicx}

\newtheorem{theorem}{Theorem}[section]
\newtheorem{lemma}[theorem]{Lemma}
\newtheorem{cor}[theorem]{Corollary}

\theoremstyle{definition}
\newtheorem{definition}[theorem]{Definition}

\theoremstyle{remark}

\numberwithin{equation}{section}

\begin{document}

\title[Classification of low-dimensional subdivision rules]{Classification of subdivision rules for geometric groups of low dimension}


\author{Brian Rushton}
\address{Department of Mathematics, Temple
University, Philadelphia, PA 19122, USA}
\email{brian.rushton@temple.edu}

\date{08/07/2014}

\begin{abstract}
Subdivision rules create sequences of nested cell structures on CW-complexes, and they frequently arise from groups. In this paper, we develop several tools for classifying subdivision rules. We give a criterion for a subdivision rule to represent a Gromov hyperbolic space, and show that a subdivision rule for a hyperbolic group determines the Gromov boundary. We give a criterion for a subdivision rule to represent a Euclidean space of dimension less than 4. We also show that Nil and Sol geometries can not be modeled by subdivision rules. We use these tools and previous theorems to classify the geometry of subdivision rules for low-dimensional geometric groups by the combinatorial properties of their subdivision rules.
\end{abstract}

\maketitle
\section{Introduction}\label{Introduction}
Finite subdivision rules provide a relatively new technique in geometric group theory for studying the quasi-isometry properties of groups \cite{cannon2001finite,hyperbolic}. A finite subdivision rule is a way of recursively defining an infinite sequence of nested coverings of a cell complex, usually a sphere. The sequence of coverings can be converted into a graph called the history graph (similar to the history complex described in p. 30 of \cite{cannon1999conformal}). We typically construct history graphs that are quasi-isometric to a Cayley graph of a group. The quasi-isometry properties of the group are then determined by the combinatorial properties of the subdivision rule.

Cannon and others first studied subdivision rules in relation to hyperbolic 3-manifold groups \cite{hyperbolic}. Cannon, Floyd, Parry, and Swenson found necessary and sufficient combinatorial conditions for a two-dimensional subdivision rule to represent a hyperbolic 3-manifold group (see Theorem 2.3.1 of \cite{hyperbolic}, the main theorem of \cite{conformal}, and Theorem 8.2 of \cite{rich}).

This paper expands those results by finding conditions for subdivision rules to represent other important geometries, including all those of dimension 3 or less.

Our main results are the following (some terms will be defined later):

\newtheorem*{HyperHausdorff}{Theorem \ref{HyperHausdorff}}
\begin{HyperHausdorff}
Let $R$ be a subdivision rule, and let $X$ be a finite $R$-complex. If $R$ is hyperbolic with respect to $X$, then the history graph $\Gamma=\Gamma(R,X)$ is Gromov hyperbolic.
\end{HyperHausdorff}

\newtheorem*{HyperQuotient}{Theorem \ref{HyperQuotient}}
\begin{HyperQuotient}
Let $R$ be a subdivision rule and let $X$ be a finite $R$-complex. If $\Gamma(R,X)$ is $\delta$-hyperbolic, then $R$ is hyperbolic with respect to $X$ and the canonical quotient $\widehat{\Lambda}$ is homeomorphic to the Gromov boundary $\partial \Gamma$ of the history graph. The preimage of each point in the quotient is connected, and its combinatorial diameter in each $\Lambda_n$ has an upper bound of $\delta+1$.
\end{HyperQuotient}

\newtheorem*{CombableThm}{Theorem \ref{CombableThm}}
\begin{CombableThm}
History graphs of subdivision rules are combable. Thus, groups which are quasi-isometric to history graphs are combable and have a quadratic isoperimetric inequality.
\end{CombableThm}

\newtheorem*{EuclideanThm}{Theorem \ref{EuclideanTheorem}}
\begin{EuclideanThm}
Let $R$ be a subdivision rule, and let $X$ be a finite $R$-complex. Assume that $\Gamma(R,X)$ is quasi-isometric to a Cayley graph of a group $G$. Then:
\begin{enumerate}
\item the counting function $c_X$ is a quadratic polynomial if and only if $\Gamma(R,X)$ is quasi-isometric to $\mathbb{E}^2$, the Euclidean plane, and
\item the counting function $c_X$ is a cubic polynomial if and only if $\Gamma(R,X)$ is quasi-isometric to $\mathbb{E}^3$, Euclidean space.
\end{enumerate}
\end{EuclideanThm}

These four theorems allow us to distinguish easily between subdivision rules representing a group with a 2-dimensional geometry or a 3-dimensional geometry. We discuss each geometry in detail in Section \ref{ClassificationSection}.

\subsection{Acknowledgements}
The author would like to thank Ruth Charney for introducing him to many quasi-isometry properties, and Jason Behrstock for introducing him to many of the properties of combable spaces that were used in this paper.

The anonymous referee provided numerous helpful suggestions leading to substantial revision.

\subsection{Outline}

Our goal is to identify subdivision rules and complexes whose history graphs are quasi-isometric to important spaces such as hyperbolic space or Euclidean space. They allow us to convert many questions about quasi-isometries into questions about combinatorics.

In Section \ref{BackgroundSection}, we introduce the necessary background material on quasi-isometries and finite subdivision rules.

In Section \ref{HyperSection}, we establish the relationship between subdivision rules and the Gromov boundary of hyperbolic spaces. In particular, we give a simple combinatorial characterization of those subdivision rules and complexes whose history graphs are hyperbolic, and show how to recover the Gromov boundary of such a history graph directly from the subdivision rule and complex.

In Section \ref{Combable}, we show that all groups quasi-isometric to a history graph satisfy a quadratic isoperimetric inequality, which shows that 3-dimensional Nil and Sol groups cannot be quasi-isometric to a history graph.

Finally, in Section \ref{ClassificationSection}, we collect our various tools to distinguish between all history graphs which are quasi-isometric to one of the geometries of dimension less than 4. Several additional theorems are proved characterizing spherical and Euclidean geometries.

\section{Background and Definitions}\label{BackgroundSection}

\subsection{Quasi-isometries and their properties}\label{QuasiSection}

Quasi-isometries are one of the central topics in geometric group theory \cite{drutu2011lectures,de2000topics,gromov1984infinite}. A quasi-isometry is a looser kind of map than an isometry; instead of requiring distances to be preserved, we instead require that distances are distorted by a bounded amount:

\begin{definition}
Let $f$ be a map from a metric space $X$ to a metric space $Y$. Then $f$ is a \textbf{quasi-isometry} if there is a constant $K$ such that:
\begin{enumerate}
\item for all $x_1,x_2$ in $X$, $$\frac{1}{K}d_X(x_1,x_2)-K \leq d_Y(f(x_1),f(x_2))\leq K d_X(x_1,x_2)+K$$
\item for all $y$ in $Y$, there is an $x$ in $X$ such that $d_Y(y,f(x))<K$.
\end{enumerate}
A map satisfying 1 but not 2 is called a \textbf{quasi-isometric embedding}.
\end{definition}

Quasi-isometries can be best understood by seeing what properties they preserve. One of the most important properties preserved by quasi-isometries is \textbf{hyperbolicity} \cite{gromov1987hyperbolic}. Intuitively, Gromov hyperbolic spaces are similar to trees (recall that a \textbf{tree} is a graph with no cycles). In fact, small subsets of Gromov hyperbolic spaces can be closely approximated by trees (see Theorem 1 of Chapter 6 of \cite{coornaert1990geometrie}). A rigorous definition for a Gromov hyperbolic space can be given by the Gromov product (see Section 2.10 of \cite{2005187}):

\begin{definition}
Let $X$ be a metric space. Let $x,y,z$ be points in $X$. Then the \textbf{Gromov product} of $y$ and $z$ with respect to $x$ is

$$(y,z)_x=\frac{1}{2}(d(x,y)+d(x,z)-d(y,z)).$$

By the triangle inequality, the Gromov product is always nonnegative. In a tree, the Gromov product measures how long the geodesics (i.e. shortest paths) from $x$ to $y$ and from $x$ to $z$ stay together before diverging.

Now assume $X$ is a geodesic metric space (so that the distance between points is realized by minimal-length paths). Then $X$ is $\delta$-\textbf{hyperbolic} (\textbf{Gromov hyperbolic}) if there is a $\delta>0$ such that for all $x,y,z,p$ in $X$,

$$(x,z)_p > \min\{(x,y)_p,(y,z)_p\}-\delta$$
\end{definition}

An equivalent definition \cite{gromov1987hyperbolic} of a hyperbolic space is that geodesic triangles are $\delta$-\textbf{thin}, meaning that any edge in a geodesic triangle is contained in the $\delta$-neighborhood of the other 2 edges of the geodesic triangle.

Being hyperbolic is a quasi-isometry invariant \cite{gromov1987hyperbolic}. Another quasi-isometry invariant of a hyperbolic group is its boundary \cite{kapovich2002boundaries}:

\begin{definition}
Let $X$ be a metric space with a fixed basepoint $O$. Let $Z$ denote the set of infinite geodesic rays starting from $O$, parametrized by arclength. Define an equivalence relation on $Z$ by letting $\gamma \sim \nu$ if the Gromov product $(\gamma(t),\nu(t))_O$ diverges to infinity as $t$ goes to infinity.

The set of all such equivalence classes is denoted $\partial X$, and is called the \textbf{boundary} of $X$. Its topology is generated by basis elements of the form

$$U(\gamma,r)=\{[\nu]\in \partial X|(\gamma(t),\nu(t))_O\geq r \text{ for all $t$ sufficiently large} \}$$

\end{definition}

A quasi-isometry between metric spaces $X$ and $Y$ induces a homeomorphism from $\partial X$ to $\partial Y$ (Proposition 2.20 of \cite{kapovich2002boundaries}). The subdivision rules we describe in this paper play a similar role to the boundary of a hyperbolic group, as we shall see. Subdivision rules are not quasi-isometry invariants, but many of their combinatorial properties are.

Two other invariants we will use are growth and ends.

\begin{definition}
A \textbf{growth function} for a discrete metric space $X$ with base point $x$ and finite metric balls is a function $g:\mathbb{N}\rightarrow\mathbb{N}$ such that $g(n)$ is the number of elements of $X$ of distance no more than $n$ from the basepoint $x$.
\end{definition}

The degree of the growth function (polynomial of degree $d$, exponential, etc.) is a quasi-isometry invariant (Lemma 12.1 of \cite{drutu2011lectures}). If a growth function for a space is a polynomial in $n$, we say that that space has \textbf{polynomial growth}. If a space has a growth function that is a quadratic or cubic polynomial, we say that the space has \textbf{quadratic growth} or \textbf{cubic growth}, respectively.

\begin{definition} Let $X$ be a metric space. Then an \textbf{end} of $X$ is a sequence  $E_1\subseteq E_2 \subseteq E_3\subseteq...$ such that each $E_n$ is a component of $X\setminus B(0,n)$, the complement in $X$ of a ball of radius $n$ about the origin.

If $G$ is a group, then an \textbf{end} of $G$ is an end of a Cayley graph for $G$ with the word metric.
\end{definition}

The cardinality of the set of ends is a quasi-isometry invariant (Proposition 6.6 of \cite{drutu2011lectures}).

\subsection{Subdivision rules}

Subdivision rules arise frequently in mathematics in many guises. They are rules for repeatedly dividing a topological object into smaller pieces in a recursively defined way. Barycentric subdivision, the middle thirds construction of the Cantor set (and its analogues for the Sierpinski carpet and Menger sponge), binary subdivision (used in the proof of the Heine-Borel theorem), and hexagonal refinement (used in circle packings; see p.158 of \cite{stephenson2005introduction}) are all examples of subdivision rules used in mathematics.

The most commonly studied type of subdivision rule is a \textbf{finite subdivision rule} \cite{cannon2001finite,CubeSubs}. Intuitively, a finite subdivision rule takes a CW-complex where each cell is labelled and refines each cell into finitely many smaller labelled cells according to a recursive rule. The different labels are called \textbf{tile types}. If two tiles have the same type, they are subdivided according to the same rule. Each tile type is classified as \textbf{ideal} or \textbf{non-ideal}. We require that ideal tiles only subdivide into other ideal tiles. This distinction will play a role similar to the distinction between the limit set of a group of hyperbolic isometries and its domain of discontinuity. The rigorous definitions will be delayed to Section \ref{FormalSection}.

Given a subdivision rule (typically denoted by the letter $R$), an $R$-\textbf{complex} is essentially a CW complex consisting of a number of top-dimensional cells $T_1,...,T_n$ labelled by tile types. We define $R(X)$ to be the union of the subdivisions $R(T_1)$,...,$R(T_n)$. We can subdivide again to get $R(R(X))$, which we write as $R^2(X)$. We can continue to define $R^3(X),R^4(X)$, etc.

Barycentric subdivision in dimension $n$ is the classic example of a subdivision rule. There is only one tile type (a simplex of dimension $n$), and each simplex of dimension $n$ is subdivided into $(n+1)!$ smaller simplices of dimension $n$.

Subdivision rules were originally used to study hyperbolic 3-manifold groups, and only 2-dimensional subdivision rules were considered \cite{cannon2001finite,Expansion1,Expansion2}.

\subsection{The history graph}\label{HistorySection}

The history graph is one of the most useful constructions involving subdivision rules. It is a metric space whose quasi-isometry properties are directly determined by the combinatorial properties of a given subdivision rule $R$ and an $R$-complex $X$.

We require a preliminary definition:

\begin{definition}
Let $R$ be a subdivision rule, and let $X$ be an $R$-complex. The interior of the union in $X$ of all ideal tiles in every level $R^n(X)$ is called the \textbf{ideal set} and is denoted $\Omega=\Omega(R,X)$. Its complement is called the \textbf{limit set} and is denoted $\Lambda=\Lambda(R,X)$. We will use $\Lambda_n$ to denote the union of all non-ideal tiles in the $n$th level of subdivision $R^n(X)$.
\end{definition}

\begin{definition} Let $R$ be a subdivision rule, and let $X$ be an $R$-complex. Let $\Gamma_n$ be a graph with

\begin{enumerate}
\item a vertex for each cell of $\Lambda_n$, and
\item an edge for each inclusion of cells of $\Lambda_n$ (i.e. if a cell $K$ is contained a larger cell $K'$, the vertex corresponding to $K$ is connected by an edge to the vertex corresponding to $K'$).
\end{enumerate}

We use $d_n$ to denote the path metric on $\Gamma_n$. The metric will take on infinite values if $\Gamma_n$ has more than one component.

The \textbf{history graph} $\Gamma=\Gamma(R,X)$ consists of:
\begin{enumerate}
\item a single vertex $O$ called the \textbf{origin},
\item the disjoint union of the $\Gamma_n$, whose edges are called \textbf{horizontal}, and
\item a collection of \textbf{vertical} edges induced by subdivision; i.e., if a vertex $v$ in $\Gamma_n$ corresponds to a cell $T$, we add an edge connecting $v$ to the vertices of $\Gamma_{n+1}$ corresponding to each of the open cells contained in the interior of $R(T)$. We also connect the origin $O$ to every vertex of $\Gamma_0$.
\end{enumerate}
\end{definition}

\begin{figure}
\begin{center}
\scalebox{.35}{\includegraphics{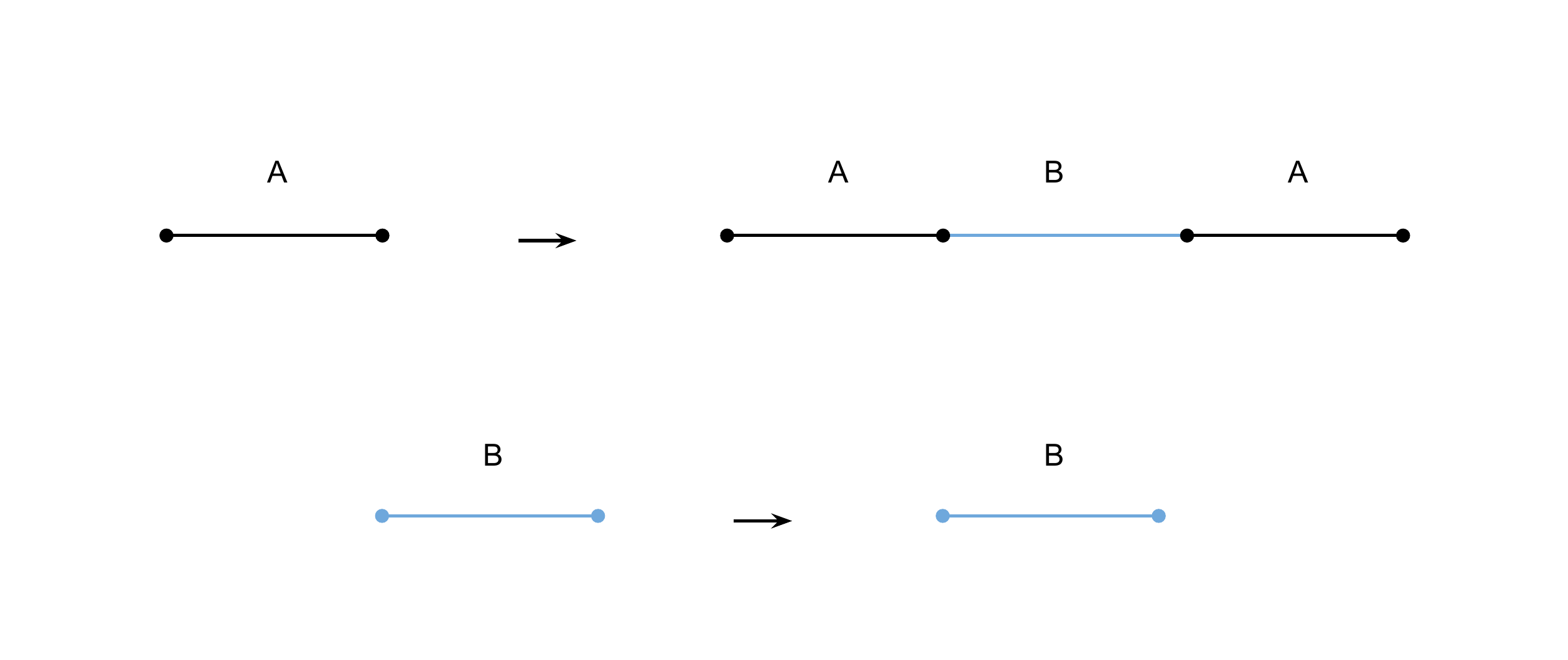}} \caption{The middle thirds subdivision rule used to create the Cantor set. The B tiles are ideal.}
\label{MiddleThirds}
\end{center}
\end{figure}

Every vertex of $\Gamma_{n+1}$ is connected by a unique edge to a vertex of $\Gamma_n$. Notice that the history graph essentially ignores ideal tiles. Ideal tiles are motivated by classic constructions such as the middle thirds subdivision rule for the Cantor set (see Figures \ref{MiddleThirds} and \ref{CantorHistoryGraph}). Labelling tiles as `ideal' is intended to mimic deleting the tiles, which is why ideal tiles are not included in the history graph.

We now define various projection functions involving the limit set and/or the history graph.

\begin{definition}
Let $x$ be an element of $\Lambda$. Because $\Lambda\subseteq \Lambda_n$ for each $n$, $x$ lies in a unique minimal cell $K$ of $\Lambda_n$. The \textbf{projection function} $f_n$ sends $x$ to the vertex of $\Gamma_n$ corresponding to $K$.
\end{definition}

\begin{definition}
Let $m<n$. If a vertex $a$ in $\Gamma_n$ corresponds to a cell $K$ in $\Lambda_n$, then let $b$ be the vertex of $\Gamma_m$ corresponding to the unique cell $K'$ of minimum dimension in $\Lambda_m$ containing $K$. Then we define the \textbf{transition function} or \textbf{projection} $f_{m,n}:\Gamma_n \rightarrow \Gamma_m$ by $f_{m,n}(a)=b$. We extend the function to the edges of $\Gamma_n$ in the natural way. The following lemma shows that another way to view transition functions is that each vertex $a$ in $\Gamma_n$ is sent to the unique vertex $b$ of $\Gamma_m$ that intersects the geodesic $Oa$.
\end{definition}

\begin{lemma}\label{GeoLemma}
Each point $x$ of $\Lambda$ corresponds to a unique geodesic ray of $\Gamma$ based at $O$. Also, every geodesic ray of $\Gamma$ based at $O$ has this form or lies within a neighborhood of radius 2 about a ray of this form.
\end{lemma}
\begin{proof}
Consider the sequence of vertices of $\Gamma$ given by the origin followed by $\{f_n(x)\}$. We claim that each vertex $f_n(x)$ is connected to $f_{n+1}(x)$ by a vertical edge. To see this, note that the minimal cell of $\Lambda_{n+1}$ that contains $x$ is a subset of the minimal cell of $\Lambda_n$ that contains $x$. Thus, there is a path $\gamma_x$ in $\Gamma$ consisting entirely of vertical edges whose vertex set is the origin together with $\{f_n(x)\}$. This path $\gamma_x$ is a geodesic, because the distance from the origin to $f_n(x)$ is $n$, which is the length of the segment of $\gamma_x$ from the origin to $f_n(x)$.

To prove the second statement, let $\gamma$ be an infinite geodesic ray based at $O$. Such a ray can contain only vertical edges; if the ray contained a horizontal edge, it could be shortened because the vertices of the horizontal edge have the same distance from the origin.

Thus, omitting the origin, the set of vertices crossed by $\gamma$ has the form $\{v_n\}$, with each $v_n$ in $\Gamma_n$ and with $v_n$ connected by a vertical edge to $v_{n+1}$. Let $K_n$ be the cell of $\Lambda^n(X)$ corresponding to the vertex $v_n$. Because each $v_n$ is connected by a vertical edge to $v_{n+1}$, we have $K_{n+1}\subseteq K_n$. Since each cell is compact and connected, the intersection $\bigcap K_n$ is nonempty.

Let $x$ be a point of this intersection. Then $x$ lies in each $K_n$; however, $K_n$ is not necessarily the \emph{minimal} cell of $R^n(X)$ containing $x$. Let $L_n$ be the minimal cell of $R^n(X)$ containing $x$, and let $w_n$ be the corresponding vertex of $\Gamma_n$. Then $L_n\subseteq K_n$, and so each $w_n$ is connected to $v_n$ by a horizontal edge. By the first portion of the proof, there is a geodesic $\gamma'$ going through all of the $w_n$. Thus, all points of $\gamma$ (including the points on the edges) lie within the neighborhood of radius 2 about $\gamma'$.
\end{proof}

\begin{figure}
\begin{center}
\scalebox{.35}{\includegraphics{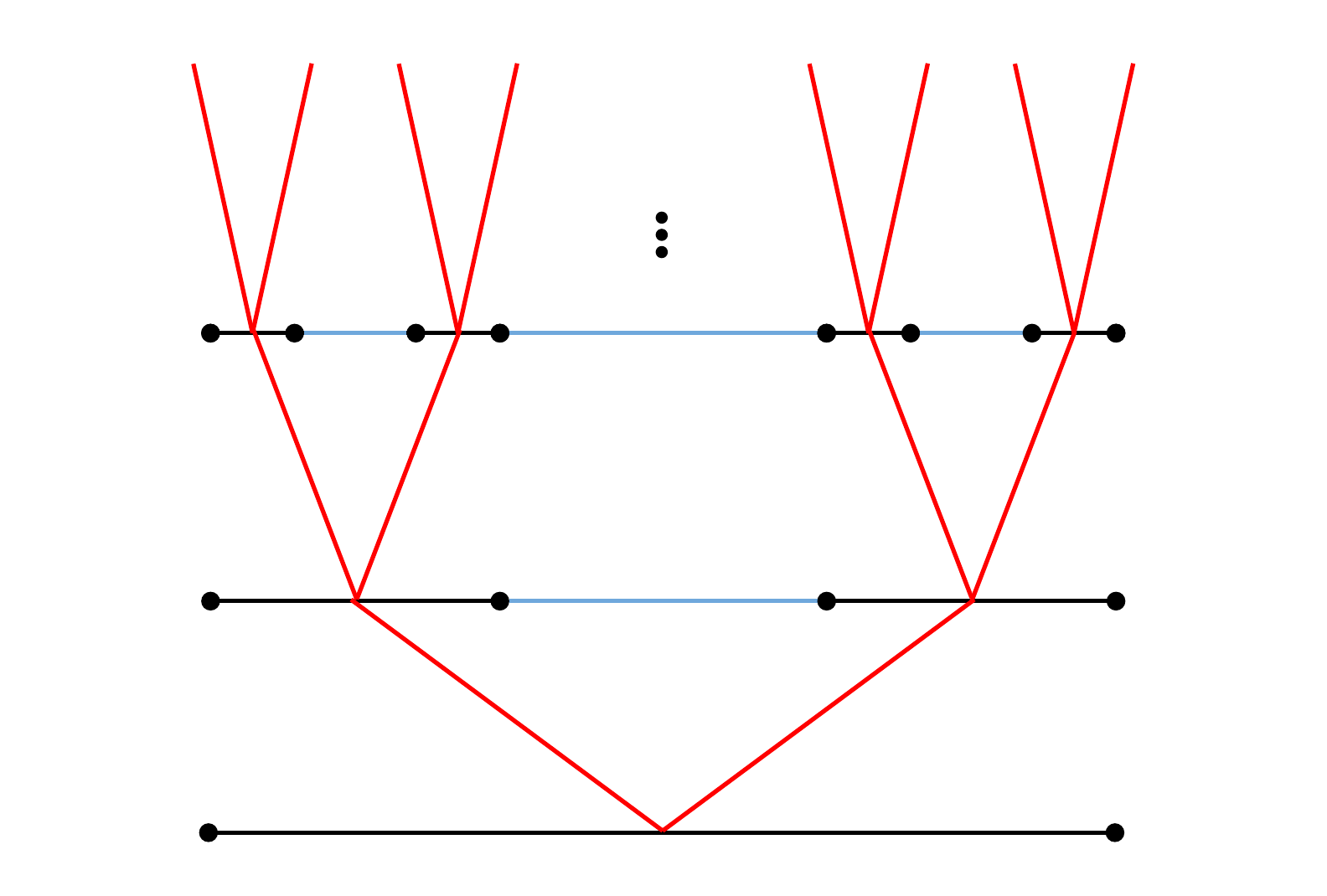}} \caption{The history graph of the middle thirds subdivision rule $R$ and $R$-complex $X$ consisting of a single tile of type $A$. The origin $O$ is omitted.}
\label{CantorHistoryGraph}
\end{center}
\end{figure}

The quasi-isometry properties of the history graph are determined by the combinatorial properties of the subdivision rule. For instance, we have the following theorems:

\begin{theorem}\label{Quasi3Thm}
Let $R$ be a subdivision rule, and let $X$ be an $R$-complex. Let $Y$ be a metric space that is quasi-isometric to the history graph $\Gamma(R,X)$. Then the cardinality of the set of ends of $Y$ is the same as the cardinality of the set of components of $\Lambda$.
\end{theorem}
\begin{proof}
Let $B$ be a component of $\Lambda$. Then $B$ is contained in a component $B_n$ of $\Lambda_n$. The set $B_n$ is connected and locally path connected, so it is path connected. This means that it corresponds to a connected subgraph $\Delta_n$ of $\Gamma_n$. Let $x,y$ be points in $B$. By Lemma \ref{GeoLemma}, there are geodesic rays $\gamma_x$ and $\gamma_y$ from the origin which correspond to $x$ and $y$, so that $\gamma_x(n)=f_n(x)$ and $\gamma_y(n)=f_n(y)$. For each $n$, the points $\gamma_x(n),\gamma_y(n)$ lie in $\Delta_n$. Thus, there is a path $\alpha_n$ connecting $\gamma_x(n)$ and $\gamma_y(n)$. The paths $\alpha_n$ lie in the $n$-sphere in $\Gamma$, which lies outside the ball $B(0,n)$ of radius $n$ about the origin. Thus, the rays corresponding to $x$ and $y$ are connected by paths lying outside of $B(0,m)$ for any $m$, so are in the same end of $\Gamma$.

Thus, all rays corresponding to points of $B$ are in the same end of $\Gamma$.

On the other hand, given an end $\{E_n\}$ of $\Gamma$, let $x,y$ be points in $\Lambda$ such that the geodesic rays $\gamma_x$ and $\gamma_y$ corresponding to $x$ and $y$ remain in $\{E_n\}$ (such points exist by Lemma \ref{GeoLemma}). Then for each $n$, there is a path $\alpha_n$ in $\Gamma\setminus B(0,n)$ connecting $f_n(x)$ to $f_n(y)$. Each $\alpha_n$ can be projected to a path in $\Gamma_n$ by transition functions. More explicitly, this is done by mapping every vertex of $\alpha_n$ to $\Gamma_n$ by the appropriate transition functions, and then sending all edges in the path to the corresponding edge between the projections of its vertices, or to a single point if the endpoints are identified. The existence of $\alpha_n$ implies that there is a chain of cells of $\Lambda_n$ connecting $x$ and $y$. This implies that $x$ and $y$ lie in the same component $B_n$ of $\Lambda_n$. Because the intersection of a nested sequence of compact connected sets is connected, $x$ and $y$ lie in the same component of $\Lambda$.

Thus, the cardinality of the set of ends of $\Gamma$ is equal to the cardinality of the set of components of $\Lambda$. If $Y$ is a space quasi-isometric to $\Gamma$, then by Proposition 6.6 of \cite{drutu2011lectures}, their ends are in bijective correspondence.
\end{proof}

We now discuss growth, as defined in Section \ref{QuasiSection}. Two functions $f,g:\mathbb{N}\rightarrow\mathbb{N}$ are \textbf{equivalent} if there are constants $m_1,m_2,b_1,b_2$ such that $f(x)\leq
g(m_1x+b_1)+m_1x+b_1$ and $g(x)\leq
f(m_2x+b_2)+m_2x+b_2$. A growth function for a group is defined to be the growth function of the vertex set of one of its Cayley graphs.

\begin{definition}
Let $R$ be a subdivision rule and let $X$ be a finite $R$-complex. The \textbf{counting function} for $X$ is the function $c_X:\mathbb{N}\rightarrow\mathbb{N}$ whose value at $n$ is the sum of the number of cells in $\Lambda_i$ for $i\leq n$.
\end{definition}

\begin{theorem}\label{Quasi1Thm}
Let $R$ be a subdivision rule and let $X$ be a finite $R$-complex. Let $G$ be a finitely generated group with a Cayley graph which is quasi-isometric to $\Gamma=\Gamma(R,X)$. Then the growth function of $G$ is equivalent to the counting function of $X$.
\end{theorem}
\begin{proof}
By construction, cells of $\Lambda_n(X)$ are in 1-1 correspondence with the vertices of the sphere of radius $n$ in the history graph $\Gamma$. The history graph $\Gamma$ is quasi-isometric to a Cayley graph of $G$. By Lemma 5.1 of \cite{drutu2011lectures}, the growth rate of $G$ is equivalent to the growth rate of $\Gamma$, which is the sum of the number of tiles in $\Lambda_i$ for $i\leq n$.
\end{proof}

As we will show in Section \ref{EuclideanSection}, the above theorem can be used to give a criterion for a history graph to be quasi-isometric to the Euclidean plane or to Euclidean space.

\subsection{Formal Definition of a Subdivision Rule}\label{FormalSection}

At this point, it may be helpful to give a concrete definition of
subdivision rule. This definition is highly abstract and may be omitted on the first reading.

Cannon, Floyd and Parry gave the first
definition of a finite subdivision rule (for instance, in
\cite{cannon2001finite}); however, their definition only applies to
subdivision rules on 2-complexes. In this paper, we study more general subdivision rules. A
\textbf{(colored) finite subdivision rule $R$ of dimension $n$} consists of:
\begin{enumerate} \item A finite $n$-dimensional CW complex $S_R$,
called the \textbf{subdivision complex}, with a fixed cell structure
such that $S_R$ is the union of its closed $n$-cells (so that the complex is pure dimension $n$). We assume that for
every closed $n$-cell $\tilde{s}$ of $S_R$ there is a CW structure $s$
on a closed $n$-disk such that any two subcells that intersect do so in
a single cell of lower dimension, the subcells of $s$ are contained in
$\partial s$, and the characteristic map $\psi_s:s\rightarrow S_R$ which
maps onto $\tilde{s}$ restricts to a homeomorphism onto each open cell.
\item A finite $n$-dimensional complex $R(S_R)$ that is a subdivision of $S_R$.
\item A \textbf{coloring} of the tiles of $S_R$, which is a partition of the set of tiles of $S_R$ into an ideal set $I$ and a non-ideal set $N$.
\item A \textbf{subdivision map} $\phi_R: R(S_R)\rightarrow S_R$, which is a
continuous cellular map that restricts to a homeomorphism on each open
cell, and which maps the union of all tiles of $I$ into itself. \end{enumerate}

Each cell $s$ in the definition above (with its appropriate
characteristic map) is called a \textbf{tile type} of $S_R$. We will often describe an
$n$-dimensional finite subdivision rule by the subdivision of every tile
type, instead of by constructing an explicit complex.

Given a finite subdivision rule $R$ of dimension $n$, an $R$-\textbf{complex}
consists of an $n$-dimensional CW complex $X$ which is the union of its
closed $n$-cells, together with a continuous cellular map $f:X\rightarrow
S_R$ whose restriction to each open cell is a homeomorphism. All tile
types with their characteristic maps are $R$-complexes.

We now describe how to subdivide an $R$-complex $X$ with map
$f:X\rightarrow S_R$, as described above. Recall that $R(S_R)$ is a
subdivision of $S_R$. We simply pull back the cell structure on $R(S_R)$
to the cells of $X$ to create $R(X)$, a subdivision of $X$. This gives
an induced map $f:R(X)\rightarrow R(S)$ that restricts to a
homeomorphism on each open cell. This means that $R(X)$ is an
$R$-complex with map $\phi_R \circ f:R(X)\rightarrow S_R$. We can
iterate this process to define $R^n(X)$ by setting $R^0 (X) =X$ (with
map $f:X\rightarrow S_R$) and $R^n(X)=R(R^{n-1}(X))$ (with map $\phi^n_R
\circ f:R^n(X)\rightarrow S_R$) if $n\geq 1$.

We will use the term `subdivision rule' throughout to mean a colored finite
subdivision rule of dimension $n$ for some $n$. As we said earlier, we
will describe an $n$-dimensional finite subdivision rule by a
description of the subdivision of every tile type, instead of by
constructing an explicit complex.

\section{Hyperbolic subdivision rules and the Gromov boundary}\label{HyperSection}

In this section, we will prove Theorems \ref{HyperHausdorff} and \ref{HyperQuotient}, which characterize subdivision rules and complexes whose history graphs are Gromov-hyperbolic and shows their relationship with the Gromov boundary.

In the remainder of the paper, we let $d_n$ denote the metric on $\Gamma_n$.

\begin{definition}
Let $R$ be a finite subdivision rule and let $X$ be an $R$-complex. We say that $R$ is \textbf{hyperbolic with respect to} $X$ if there are positive integers $M,j$ such that every pair of points $x$,$y$ in $\Gamma_{n+j}$ that satisfy
$$d_{n+j}(x,y)<\infty$$ and
$$d_n(f_{n,n+j}(x),f_{n,n+j}(y))\geq M$$ also satisfy $$d_{n+j}(x,y)>d_n(f_{n,n+j}(x),f_{n,n+j}(y)).$$
\end{definition}

Now, we define a standard path.

\begin{definition}
Let $R$ be a subdivision rule and let $X$ be an $R$-complex. Assume that $R$ is hyperbolic with respect to $X$, with constant $M$ from the definition of hyperbolicity. A \textbf{standard path} in $\Gamma=\Gamma(R,X)$ from a point $x$ to a point $y$ is a geodesic that consists of a vertical, downward path beginning at $x$, a purely horizontal path of length $\leq M$, followed by an upward vertical path ending at $y$. Here `upward' is further from the origin and `downward' is closer to the origin.
\end{definition}

For the proofs of Lemma \ref{StandardLemma} and Theorem \ref{HyperHausdorff} only, we alter the metric on $\Gamma$ by letting each vertical edge have length $\frac{1}{3j}$, where $j$ is the constant from the definition of hyperbolicity for the subdivision rule in question. This changes $\Gamma$ by a quasi-isometry; it is easier to show $\Gamma$ is hyperbolic with this metric. Since hyperbolicity is a quasi-isometry invariant, this implies that $\Gamma$ with the standard metric is also hyperbolic.

\begin{lemma}\label{StandardLemma}
Let $R$ be a subdivision rule that is hyperbolic with respect to an $R$-complex $X$. Then every geodesic $\alpha$ in the history graph $\Gamma=\Gamma(R,X)$ between vertices is within $2M$ of a standard path, where $M$ is a constant depending only on $R$ and $X$.
\end{lemma}

\begin{proof}
Let $x$ in $\Gamma_m$ be the initial point of $\alpha$, and let $y$ in $\Gamma_n$ be the terminal point. Then let $A$ be the set of all vertical, downward edges of $\alpha$, $B$ the set of all horizontal edges of $\alpha$, and $C$ the set of all vertical, upward edges.

We claim that all downward edges occur in $\alpha$ before all upward edges. Assume the geodesic $\alpha$ goes up one edge, follows a horizontal path in $\Gamma_p$ for some $p$, then goes down a vertical edge. The image of the horizontal segment under $f_{p-1,p}$ in $\Gamma_{p-1}$ is no longer, and so removing the vertical segments while projecting the horizontal path to $\Gamma_{p-1}$ gives a strictly shorter path.

Note that $m-|A|=n-|C|$; call this number $h$. It represents the lowest level that the path $\alpha$ reaches. Now construct a path $\beta$ consisting of $|A|$ downward edges, a horizontal path in $\Gamma_h$ of length $\leq |B|$ from $f_{h,m}(a)$ to $f_{h,n}(b)$, and an upward path of $|C|$ edges.

The horizontal path in the middle can be taken to be the image of all elements of $B$ under the appropriate transition functions. If the distance between the endpoints of this horizontal path was $>M$, $\beta$ could be shortened by adding $j$ more vertical edges to the downward path, then following a minimum length horizontal path (which is shorter than the original horizontal path by at least 1 by hyperbolicity), going up $j$ more vertical edges, and then following the original upward path (recall that each vertical edge now has length $\frac{1}{3j}$). Because $\alpha$ is a geodesic, $\beta$ cannot be shorter than $\alpha$, so the horizontal path has length $\leq M$ and $\beta$ is a standard path.

Thus, we have shown that $\alpha$ consists of a path where all downward segments occur before all upward segments, and that all horizontal segments total no more than $M$ in length.

We now show that $\alpha$ and $\beta$ stay within $2M$ of each other, depending only on the subdivision rule.
The geodesics $\alpha$ and $\beta$ both start at $x$ in $\Gamma_m$, go downward to some points in a level $\Gamma_h$ (with $\alpha$ possibly taking horizontal detours of length $\leq M$), take a horizontal path of length $\leq M$ to some other point in $\Gamma_h$, then go upward to $y$ in $\Gamma_n$ (again with possible detours for $\alpha$ of length $\leq M$). Since distances in $\Gamma_p$ are no greater than distances in $\Gamma_q$ for $p<q$, the path $\beta$ stays within $M+j/3$ of $\alpha$ on the downward segment. Then the endpoints of the last downward edges of $\alpha$ and $\beta$ lie in $\Gamma_h$ and have distance $\leq M$. The horizontal paths in $\Gamma_h$ have length $\leq M$, and so the paths lie within $2M$ of each other at all times. By symmetry, the upward segments of $\alpha$ and $\beta$ remain within $M+j/3$ of each other. Thus, every geodesic is within $2M$ of a geodesic with the same starting points that follows a standard path.
\end{proof}

Using this last lemma, we can show that $\Gamma$ is Gromov hyperbolic by considering triangles of geodesics that follow standard paths. Standard paths are useful, because they allow us to focus on the downward and upward segments. Purely vertical geodesics connecting points in different levels are unique, and so each vertex in a level has a unique downward path coming from it. This uniqueness of vertical geodesics gives our graph a tree-like structure.

The following theorems justify our use of the term `hyperbolic' for a subdivision rule $R$ with respect to an $R$-complex.

\begin{theorem}\label{HyperHausdorff}
Let $R$ be a subdivision rule, and let $X$ be a finite $R$-complex. If $R$ is hyperbolic with respect to $X$, then the history graph $\Gamma=\Gamma(R,X)$ is Gromov hyperbolic.
\end{theorem}

\begin{proof}

\begin{figure}
\scalebox{.35}{\includegraphics{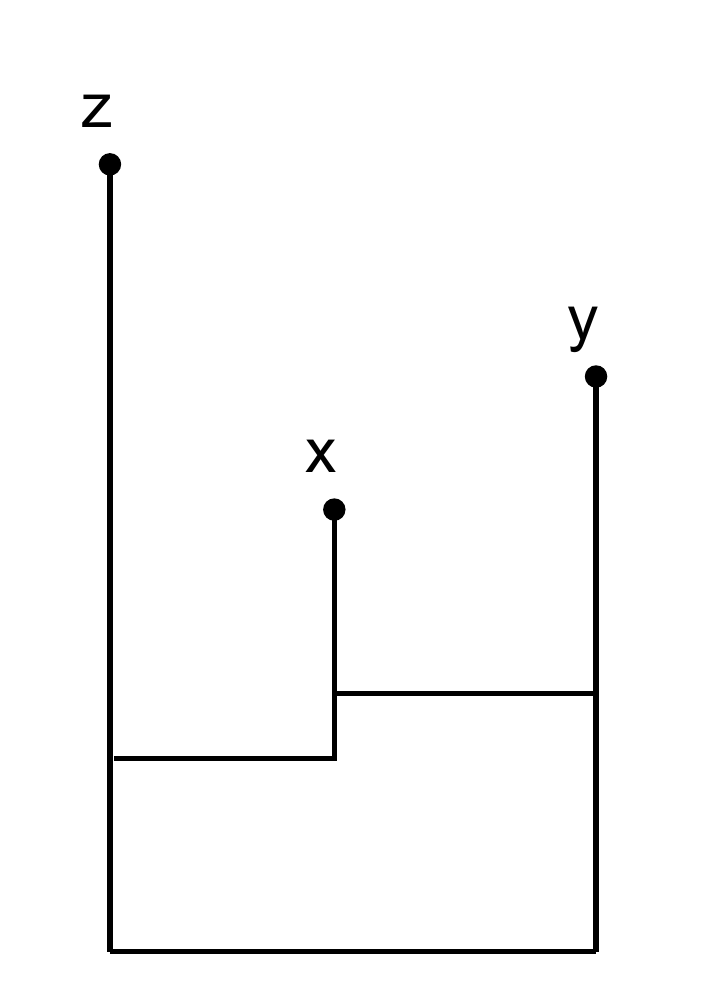}} \caption{A geodesic triangle in the history graph.}
\label{GeodesicTriangle}
\end{figure}

We will show that geodesic triangles in $\Gamma$ are $\delta$-thin. By Lemma \ref{StandardLemma}, we can assume that the geodesics in a triangle follow standard paths by changing distances a bounded amount. Now, let $x,y,$ and $z$ be points in $\Gamma$ with standard paths $\alpha_{xy},\alpha_{xz},$ and $\alpha_{yz}$ connecting them. Because these are standard paths, they are vertical except for a horizontal portion lying entirely in some level of the graph. Let $H(x,y),H(x,z)$ and $H(y,z)$ be the height of the horizontal portions of the corresponding geodesics (so, for instance, the horizontal part of $\alpha_{xy}$ lies in $\Gamma_{H(x,y)})$.

Without loss of generality, assume $H(x,y) \geq H(x,z) \geq H(y,z)$ (see Figure \ref{GeodesicTriangle}). We first show that $\alpha_{xy}$ remains close to the other two geodesics. It is the same downward path as $\alpha_{xz}$ until they reach $\Gamma_{H(x,y)}$; it then follows a path of length $\leq M$ (where $M$ is the constant from the definition of hyperbolicity), then follows an upward segment to $z$ that is the same as that followed by $\alpha_{yz}$. Thus, $\alpha_{xy}$ is within $M$ of the other paths at all times.

Now, the horizontal path of $\alpha_{xy}$ actually connects the images of $x$ and $y$ under the appropriate transition functions. So the intersections of $\Gamma_{H(x,y)}$ with the vertical segments of $\alpha_{xz}$ and $\alpha_{yz}$ are no more than $M$ apart, and their vertices in lower levels are no further apart. Thus, they are within $M$ of each other until $H(x,z)$.

In $\Gamma_{H(x,z)}$, the projections of $x$ and $z$ are $\leq M$ apart, and so are those of $y$ and $x$ (because $H(x,y)\geq H(x,z)$). Thus, the projections of $y$ and $z$ are no more than $2M$ apart. Thus, the part of $\alpha_{yz}$ which goes down, over, and up from the image of $y$ in this level to the image of $z$ in this level must be at most $2M$ in length; otherwise, the path would not be minimal length. Thus, it is never more than $3M$ away from $\alpha_{xz}$ (in fact, by symmetry, it is no more than $2M$ away). Finally, both end with the same upward segment from the projection of $z$ in $\Gamma_H(x,z)$ up to $z$ itself. Thus, all geodesics in the triangle are no more than $2M$ apart at any point. Our assumption that our geodesics were standard paths shifted each geodesic by no more than $\epsilon=2M$ (from Lemma \ref{StandardLemma}). Thus, every edge in a geodesic triangle is within $\delta=2M+2\epsilon=6M$ of the union of the other edges, and our graph $\Gamma$ is Gromov hyperbolic.
\end{proof}

Recall that the projection function $f_n$ sends each point $x$ of $\Lambda$ to the vertex of $\Gamma_n$ corresponding to the minimal cell of $\Lambda_n$ that contains $x$.

\begin{definition}
The \textbf{canonical quotient} of $\Lambda$ is the quotient given by the equivalence relation $\sim$, where $x\sim y$ if $d_n(f_n(x),f_n(y))$ is bounded as $n\rightarrow \infty$. It is denoted $\widehat{\Lambda}$.
\end{definition}

\begin{theorem}\label{HyperQuotient}
Let $R$ be a subdivision rule and let $X$ be a finite $R$-complex. If $\Gamma(R,X)$ is $\delta$-hyperbolic, then $R$ is hyperbolic with respect to $X$ and the canonical quotient $\widehat{\Lambda}$ is homeomorphic to the Gromov boundary $\partial \Gamma$ of the history graph. The preimage of each point in the quotient is connected, and its combinatorial diameter in each $\Lambda_n$ has an upper bound of $\delta+1$.
\end{theorem}
\begin{proof}
Assume that $\Gamma(R,X)$ is $\delta$-hyperbolic. Let $M\geq \delta$ and $j>3M/2+2\delta+1$ be positive integers. We claim that any two points $x,y$ of finite distance in some $\Gamma_{n+j}$ satisfying $d_n(f_{n,n+j}(x),f_{n,n+j}(y))=M$ also satisfy $d_{n+j}(x,y)>M$.

To see this, consider the geodesic triangle $\Delta$ whose vertices are the origin $O$ and the points $x$ and $y$. Assume by way of contradiction that $d_{n+j}(x,y)=M$. This implies that the projections of $x$ and $y$ in every level from $n$ to $n+j$ are $M$ apart. Then the geodesic segment between $x$ and $y$ lies entirely above $\Gamma_{n+j-M/2}$, since going down more than $M/2$ and returning again gives a path of length greater than the distance between $x$ and $y$.

Now, let $k$ be an integer strictly between $n+M$ and $n+j-M/2-\delta$. Then by the argument in the preceding paragraph, the projection $f_{k,n+j}(x)$ is not within $\delta$ of the segment between $x$ and $y$. Thus, it must be within $\delta$ of the segment from $O$ to $y$. Consider a shortest path $\alpha$ in $\Gamma$ from $f_{k,n+j}(x)$ to the segment from $O$ to $y$. The path $\alpha$ must have length $l(\alpha)<\delta\leq M$.

Let $\Gamma_{n'}$ be the lowest level that $\alpha$ reaches. Because $l(\alpha)<M$, we know that $n'>k-M> n$. Consider the projection of $\alpha$ to $\Gamma_{n'}$. The image of $\alpha$ under the projection is a path $\alpha'$ of length $l(\alpha')\leq l(\alpha) < M$. But this projected path $\alpha'$ lies in $\Gamma_{n'}$ and connects the projections $f_{n',n+j}(x)$ and $f_{n',n+j}(y)$. But because $n<n'<n+j$, we know that $d_{n'}(f_{n',n+j}(x), f_{n',n+j}(y))=M$, which is a contradiction. Thus, our assumption that $d_{n+j}(x,y)=M$ must be wrong, and it must be true that $d_{n+j}(x,y)>M$.

Thus, we have proved our claim that any two points $x,y$ in some $\Gamma_{n+j}$ of finite distance in $\Gamma_{n+j}$ which satisfy $d_n(f_{n,n+j}(x),f_{n,n+j}(y))=M$ also satisfy $d_{n+j}(x,y)>M$.

Now, we show that $R$ is hyperbolic with respect to $X$. Let $w$ and $z$ be two vertices of finite distance in some $\Gamma_{n+j}$. Assume that $d_{n}(f_{n,n+j}(w),f_{n,n+j}(z))>M$. Let $p$ be a point that lies on a shortest-length path in $\Gamma_{n+j}$ from $w$ to $z$ such that $$d_{n}(f_{n,n+j}(p),f_{n,n+j}(z))=M.$$ Then
\begin{align*}
d_{n+j}(w,z) &= d_{n+j}(w,p)+d_{n+j}(p,z)
\\&\geq d_{n}(f_{n,n+j}(w),f_{n,n+j}(p))+ d_{n+j}(p,z)
\\&> d_{n}(f_{n,n+j}(w),f_{n,n+j}(p))+d_{n}(f_{n,n+j}(p),f_{n,n+j}(z))
\\&\geq d_{n}(f_{n,n+j}(w),f_{n,n+j}(z))
\end{align*}

which shows that $R$ is hyperbolic with respect to $X$.

We now prove the second statement in the theorem, that the canonical quotient $\widehat{\Lambda}$ is homeomorphic to $\partial \Gamma$. We do this by showing that a certain quotient map from $\Lambda$ to $\partial \Gamma$ factors through $\widehat{\Lambda}$. The map $p:\Lambda\rightarrow \partial \Gamma$ sending each point $x$ in $\Lambda$ to the equivalence class of its corresponding geodesic (from Lemma \ref{GeoLemma}) is surjective, and it goes from a compact space $\Lambda$ to a Hausdorff space $\partial G$, so we need only show that it is continuous and that its fibers are the equivalence classes which define $\widehat{\Lambda}$.

To show that $p$ is continuous, we need to show that the preimage of any basis element of $\partial \Gamma$ is open. Let $x$ denote a point of $\Lambda$. Let $\gamma$ denote the geodesic ray from Lemma \ref{GeoLemma} corresponding to $x$ and let $r$ denote a real number. There is a basis for $\partial \Gamma$ consisting of sets of the form $U(\gamma,r)=\{[\nu]\in \partial\Gamma|(\gamma(t),\nu(t))_O\geq r\text{ for $t$ sufficiently large}\}$. We need only show that $p^{-1}(U(\gamma,r))$ contains an open neighborhood of $x$.

\begin{figure}
\scalebox{.35}{\includegraphics{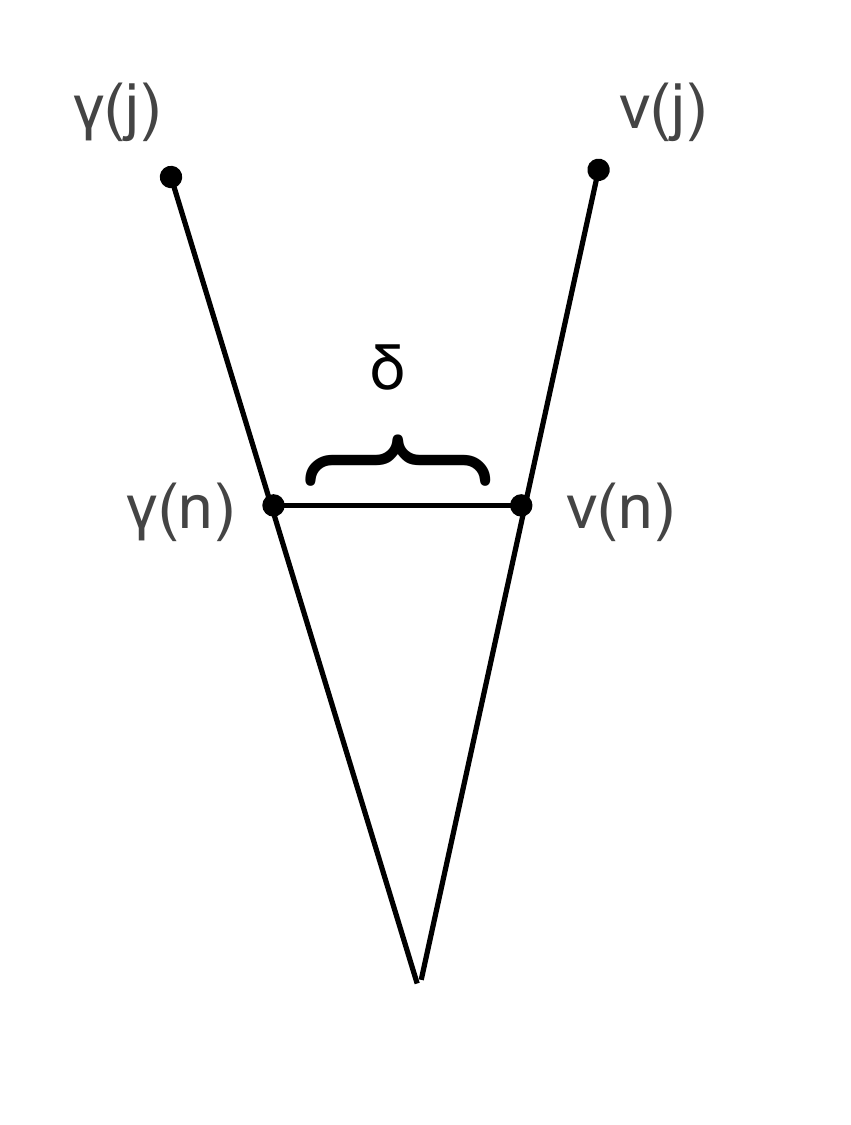}} \caption{If two geodesic rays are within $\delta$ of each other at level $n$, then the Gromov product of their endpoints at later levels does not get smaller than $n-\delta$.}
\label{TwoRays}
\end{figure}

Given an integer $k\geq \delta$, choose an integer $n\geq r+k$. Consider the set $W_{k}(\gamma,n)$ consisting of points $y\in \Lambda$ whose projection $f_n(y)$ lies within $k$ of $\gamma(n)$ in $\Gamma_n(R,X)$. Because $k$ is at least 1, the set $W_{k}(\gamma,n)$ contains an open set about $x$, namely, the \textbf{open star} about $x$ in $\Lambda_n$, consisting of the union of all open cells in $\Lambda_n$ whose closure contains $x$. We now show that $W_{k}(\gamma,n)$ lies in $p^{-1}(U(\gamma,r))$.

Consider a point of $W_{k}(\gamma,n)$, and let $\nu$ be the corresponding geodesic ray from the origin. Then $\nu(n)$ lies within $k$ of $\gamma(n)$.

Then for $j\geq n$,
$$(\gamma(j),\nu(j))_O=\frac{(d(O,\gamma(j))+d(O,\nu(j))-d(\gamma(j),\nu(j))}{2}$$

$$=\frac{d(O,\gamma(n))+d(O,\nu(n))}{2}+\frac{d(\gamma(n),\gamma(j))+d(\nu(n),\nu(j))-d(\gamma(j),\nu(j))}{2}$$

$$=n+\frac{d(\gamma(n),\gamma(j))+d(\nu(n),\nu(j))-d(\gamma(j),\nu(j))}{2}$$

But by the triangle inequality, $d(\nu(n),\nu(j))\geq d(\gamma(n),\nu(j))-d(\nu(n),\gamma(n))$. Therefore,

$$(\gamma(j),\nu(j))_O \geq n+\frac{d(\gamma(n),\gamma(j))+d(\gamma(n),\nu(j))-d(\gamma(j),\nu(j))}{2}-\frac{d(\nu(n),\gamma(n))}{2}$$
$$\geq n+(\nu(j),\gamma(j))_{\gamma(n)}-k$$
$$\geq n-k\geq r.$$

Thus, $[\nu]\in U(\gamma,r)$, and so all geodesic rays corresponding to points in $W_k(\gamma,n)$ get mapped to $U(\gamma,r)$, including the open star about $x$, and so the map $p$ is continuous as described earlier. Furthermore, since $k$ is an arbitrary integer bounded below by $\delta$, we can choose it to satisfy $\delta \leq k < \delta+1$. This proves that the combinatorial diameter of $p^{-1}([\gamma])\subseteq W_k(\gamma,n)$ is bounded above by $\delta+1$ in each $\Lambda_n(X)$.

We have $\mathop{\bigcap}\limits_n W_k (\gamma,n)\subseteq p^{-1}([\gamma])$. This implies that $p$ induces a surjective continuous map from $\widehat{\Lambda}$ to $\partial \Gamma$.

To see that this induced map is injective, again choose an integer $k\geq \delta$. Consider a point that does not lie in some $W_{k}(\gamma,N)$, and let $\alpha$ be its corresponding geodesic ray. Then for all $n\geq N$, $d_n(\gamma(n),\alpha(n))> k \geq \delta$. If this distance were bounded above, then we could choose a sufficiently large integer $M$ such that the geodesic triangle $\Delta$ with vertices the origin, $\gamma(M)$ and $\alpha(M)$ would not be $\delta$-thin, just as in the first part of the proof. Thus, the geodesic rays diverge and $[\alpha]\neq[\gamma]$. This completes the proof that $\widehat{\Lambda}$ is homeomorphic to $\partial \Gamma$.

Each of the sets $W_{k} (\gamma,n)$ with $k \geq \delta$ is connected, but $p^{-1}([\gamma])=\mathop{\bigcap}\limits_n W_{k} (\gamma,n)\subseteq \mathop{\bigcap}\limits_n \overline{W_{k} (\gamma,n)}\subseteq \mathop{\bigcap}\limits_n W_{k+1} (\gamma,n)\subseteq p^{-1}([\gamma])$.
Thus, the preimage of each point in the quotient is a connected set, as it is a nested intersection of compact connected sets.
\end{proof}

\section{Combable spaces and the isoperimetric inequality}\label{Combable}

In this section, we show that a group quasi-isometric to a history graph has a quadratic isoperimetric inequality. This will eventually be used to show that Nil and Sol manifolds cannot be modeled by finite subdivision rules.

We first define what it means for a group to satisfy an isoperimetric inequality \cite{epstein1992word}:

\begin{definition}
Let $G$ be a group with generating set $A=A^{-1}$ and relations $R$. Recall that the \textbf{length} of a reduced word $w$ in the free group $\langle A \rangle$ generated by $A$ is the number of elements required to write it. The \textbf{area} of a word $w$ in $\langle A \rangle$ that maps to the identity of $G$ is the smallest number $n$ of relators $\{r_i\}$ and words $\{g_i\}$ such that $w=\mathop{\Pi}\limits_{i=1}^n g_i^{-1}(r_i)g_i$.

The \textbf{isoperimetric function} is the function $f(n)=\max\{$area$(w)|$ $w$ maps to the identity in $G$ and length of $w=n\}$. Although the isoperimetric function itself is not a quasi-isometry invariant, its rate of growth is an invariant \cite{epstein1992word} (except in the case of constant growth and linear growth, which are equivalent to each other).

A group has a \textbf{quadratic isoperimetric inequality} if the isoperimetric function is bounded above by a quadratic polynomial.

\end{definition}

We attack the isoperimetric function for history graphs indirectly, by means of combings.

\begin{definition}
Let $P$ be the set of all paths in $\Gamma$ of the form $\gamma:[0,b]\rightarrow \Gamma$ with $\gamma(0)=O$. Let the \textbf{endpoint} of $\gamma$ be $\gamma(b)$. Assume that $\gamma_i$ is such a path with domain $[0,b_i]$ for $i=1,2$. Extend the domains of the $\gamma_i$ by letting $\gamma_i(t)=\gamma_i(b)$ for $t\geq b$ and for $i=1,2$. Define a metric on $P$ by letting $d_P(\gamma_1,\gamma_2)=|b_1-b_2| + \mathop{\max}\limits_{0\leq t< \infty}d (\gamma_1(t),\gamma_2(t))$.

Finally, the \textbf{endpoint map} sends a path $\gamma$ with domain $[0,b]$ to the endpoint $\gamma(b)$.
\end{definition}

\begin{definition} A space is \textbf{combable} if there is a right inverse to the endpoint map that is a quasi-isometric embedding.
\end{definition}

Combable groups have quadratic isoperimetric inequality (Theorem 3.6.6 of \cite{epstein1992word}). Being combable is a quasi-isometry invariant of metric spaces (Theorem 3.6.4 of \cite{epstein1992word}).
\begin{theorem}\label{CombableThm}
History graphs of subdivision rules are combable. Thus, groups which are quasi-isometric to history graphs are combable and have a quadratic isoperimetric inequality.
\end{theorem}

\begin{proof}
Send every vertex $v$ of the history graph to the unique geodesic $\gamma_v$ from the origin to $v$, parametrized by arc length. Let $v\in \Gamma_b$ and $v'\in \Gamma_b'$. Then $|b-b'|$ and $\mathop{\max}\limits_{0\leq t< \infty} d(\gamma_v(t),\gamma_{v'}(t))$ are both at most $d(v,v')$. Thus, $d_P(\gamma_v,\gamma_{v'})\leq 2d(v,v')$.

Conversely, if the distance between two paths $\gamma_v$ and $\gamma_{v'}$ is $K$, then the distance between the endpoints is at most $K$. Thus, $d(v,v')\leq d_P(\gamma_v,\gamma_{v'})$.

We can extend the map $v \mapsto \gamma_v$ to all points of $\Gamma$ by sending each point $x$ to a geodesic $\gamma_x$ from the origin to $x$ (parametrized by arc length). This path is unique except for midpoints of horizontal edges, where we can choose from two paths that are near to each other in the metric on the path space. Then each point $x$ is within a distance of $1/2$ from some vertex $v$, and the path assigned to $x$ is within a distance of 1 from the path assigned to $v$. Thus,

$$d(x,x')\leq d_P(\gamma_x,\gamma_{x'})\leq d_P(\gamma_v,\gamma_{v'})+2\leq 2d(v,v')+2\leq 2d(x,x')+4$$

Thus, the map sending $x$ to $\gamma_x$ is a right-inverse for the endpoint map and is a quasi-isometric embedding.
\end{proof}

\section{Classification of subdivision rules for low-dimensional geometries}\label{ClassificationSection}

In this section, we use the results of the earlier sections to find conditions on subdivision rules that will distinguish one low-dimensional geometric group from another. We give more general characterizations when possible.

\subsection{Compact geometries: $\mathbb{S}^2,\mathbb{S}^3$}

All compact spaces are quasi-isometric to each other. This is the geometry of finite groups.

\textbf{Example}: Let $R$ be a 0-dimensional subdivision rule with one tile type $A$ consisting of a single ideal point which subdivides into another single $A$ tile. Let $X$ be the $R$-complex consisting of a single $A$ tile. Then $\Gamma(R,X)$ is just a single vertex, the origin $O$.

\textbf{Example}: Building on the previous example, any complex with only ideal tiles has a history graph consisting of a single vertex.

There are many other subdivison rules and complexes with this geometry, which we can classify by the following:

\begin{theorem}
Let $R$ be a subdivision rule and let $X$ be an $R$-complex. Then $\Gamma(R,X)$ is compact if and only if $\Lambda(R,X)$ is empty.
\end{theorem}

\begin{proof}
This follows from Theorems \ref{HyperHausdorff} and \ref{HyperQuotient} because compact spaces are Gromov hyperbolic, and are characterized among hyperbolic spaces by their Gromov boundaries being empty.
\end{proof}

Thus, the history graph of a subdivision rule is quasi-isometric to a sphere (or to any other compact space) exactly when the limit set is empty. This is the simplest of all cases.

\subsection{Cyclic geometries: $\mathbb{R},\mathbb{S}^2\times \mathbb{R}$}

These geometries are only slightly more complicated than the compact geometries. The simplest group with this geometry is $\mathbb{Z}$, and in fact any group quasi-isometric to $\mathbb{Z}$ contains a finite index copy of $\mathbb{Z}$ \cite{farrell1995lower}.

\textbf{Example}: Similar to the previous section, we can construct examples from points that do not subdivide (in this case, two non-ideal points). Let $R$ be a 0-dimensional subdivision rule with one tile type $A$ that is non-ideal and that subdivides into another $A$ tile. Let $X$ be the $R$-complex consisting of 2 type $A$ tiles. Then $\Gamma(R,X)$ is isomorphic as a graph to the standard Cayley graph of $\mathbb{Z}$.

\begin{theorem}
Let $R$ be a subdivision rule and let $X$ be an $R$-complex. Assume that $\Gamma(R,X)$ is quasi-isometric to a group $G$. Then $\Gamma$ is quasi-isometric to $\mathbb{Z}$ if and only if:
\begin{enumerate}
\item the limit set $\Lambda$ has 2 components, and
\item the diameters of components of $\Lambda_n$ are globally bounded.
\end{enumerate}
\end{theorem}

\begin{proof}
$\implies$
The Gromov boundary of $\mathbb{Z}$ consists of two points, so the canonical quotient of $\Lambda$ must be two points, and by Theorem \ref{HyperQuotient} the pre-image of each point (i.e. each of the two components) must have bounded diameter, where the bound does not depend on the level of subdivision.

$\impliedby$
Since the diameter of every component of $\Lambda_n$ is bounded, the subdivision rule $R$ is trivially hyperbolic with respect to $X$, and since the components are preserved in the canonical quotient, the Gromov boundary consists of 2 points. Thus, $\Gamma$ is quasi-isometric to the integers.
\end{proof}

\subsection{Hyperbolic geometries: $\mathbb{H}^2$ and $\mathbb{H}^3$}

By Theorem \ref{HyperQuotient}, if a history graph $\Gamma$ for a subdivision rule $R$ and $R$-complex $X$ is quasi-isometric to a  hyperbolic 2- or 3-manifold group, then $R$ must be hyperbolic with respect to $X$ and the canonical quotient of the limit set will be a circle or a 2-sphere, respectively.

\begin{theorem}
Let $R$ be a subdivision rule and let $X$ be an $R$-complex. Suppose a group $G$ is quasi-isometric to the history graph $\Gamma(R,X)$. Suppose that:
\begin{enumerate}
\item $R$ is hyperbolic with respect to $X$, and
\item  the canonical quotient of the limit set is a circle.
\end{enumerate}
Then the group $G$ is Fuchsian (i.e. it acts geometrically on hyperbolic 2-space). The converse also holds; if $G$ is Fuchsian and is quasi-isometric to a history graph $\Gamma(R,X)$, then $R$ is hyperbolic with respect to $X$ and the canonical quotient is a circle.
\end{theorem}

\begin{proof}
By Theorem \ref{HyperHausdorff}, the history graph $\Gamma(R,X)$ is Gromov hyperbolic and by Theorem \ref{HyperQuotient} its space at infinity is a circle. Since $G$ is quasi-isometric to $\Gamma(R,X)$, it must also be hyperbolic and have a circle at infinity. By work of various authors, including Gabai, Tukia, Freden, and Casson-Jungreis \cite{gabai1992convergence,tukia1988homeomorphic,freden1995negatively,casson1994convergence}, $G$ must be (virtually) a hyperbolic 2-manifold group.

The converse holds by Theorem $\ref{HyperQuotient}$.
\end{proof}

The corresponding result does not necessarily hold if the boundary is a 2-sphere.

\begin{theorem}Let $R$ be a subdivision rule and let $X$ be an $R$-complex. Suppose a group $G$ is quasi-isometric to the history graph $\Gamma(R,X)$. Suppose that:
\begin{enumerate}
\item $R$ is hyperbolic with respect to $X$,
\item the canonical quotient of the limit set is a sphere, and
\item $G$ is known to be a manifold group

\end{enumerate}
Then the group $G$ is Kleinian (i.e. it acts geometrically on hyperbolic 3-space).
\end{theorem}

\begin{proof}
As before, the hypotheses imply that $G$ is hyperbolic with a 2-sphere at infinity. If the group is known to be a manifold group, the Geometrization Theorem \cite{perelman2002entropy, perelman2003ricci} implies that the group is quasi-isometric to hyperbolic 3-space.
\end{proof}

There are numerous explicit examples of finite subdivision rules on the 2-sphere that represent hyperbolic 3-manifolds, as well as subdivision rules representing hyperbolic knot complements \cite{cannon2001finite, linksubs,CubeSubs,PolySubs}.

\subsection{Euclidean geometries: $\mathbb{E}^2$ and $\mathbb{E}^3$}\label{EuclideanSection}

\begin{figure}
\begin{center}
\scalebox{.4}{\includegraphics{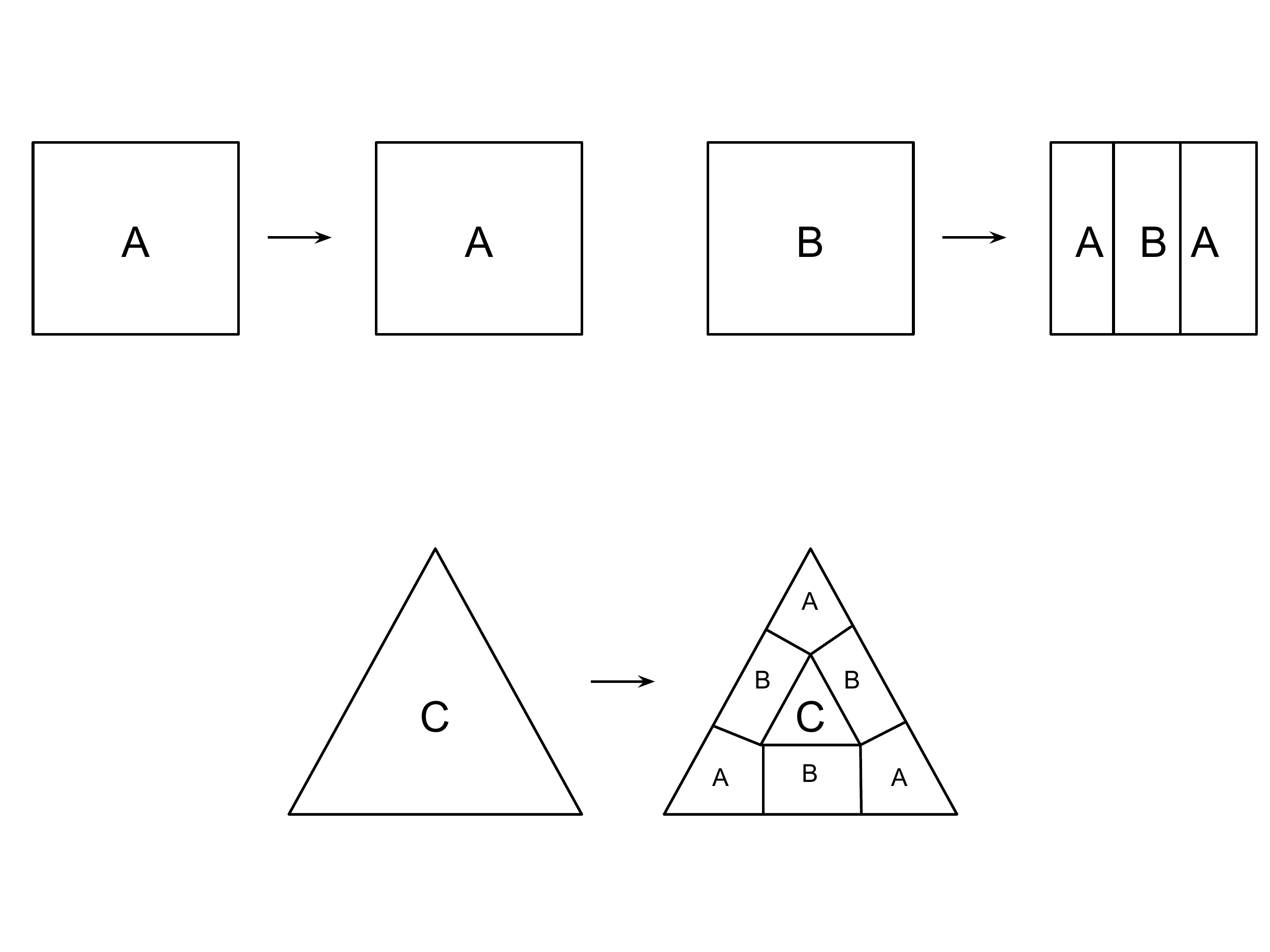}}
\caption{The tile types of a subdivision rule $R$. The actual subdivision complex $S_R$ is created by pasting all outer edges in this figure together that subdvide the same way. With the $R$-complex $X$ shown in Figure \ref{TorusS0}, the history graph $\Gamma(R,X)$ is quasi-isometric to Euclidean space.}\label{TorusSubdivision}
\end{center}
\end{figure}

\begin{figure}
\begin{center}
\scalebox{.4}{\includegraphics{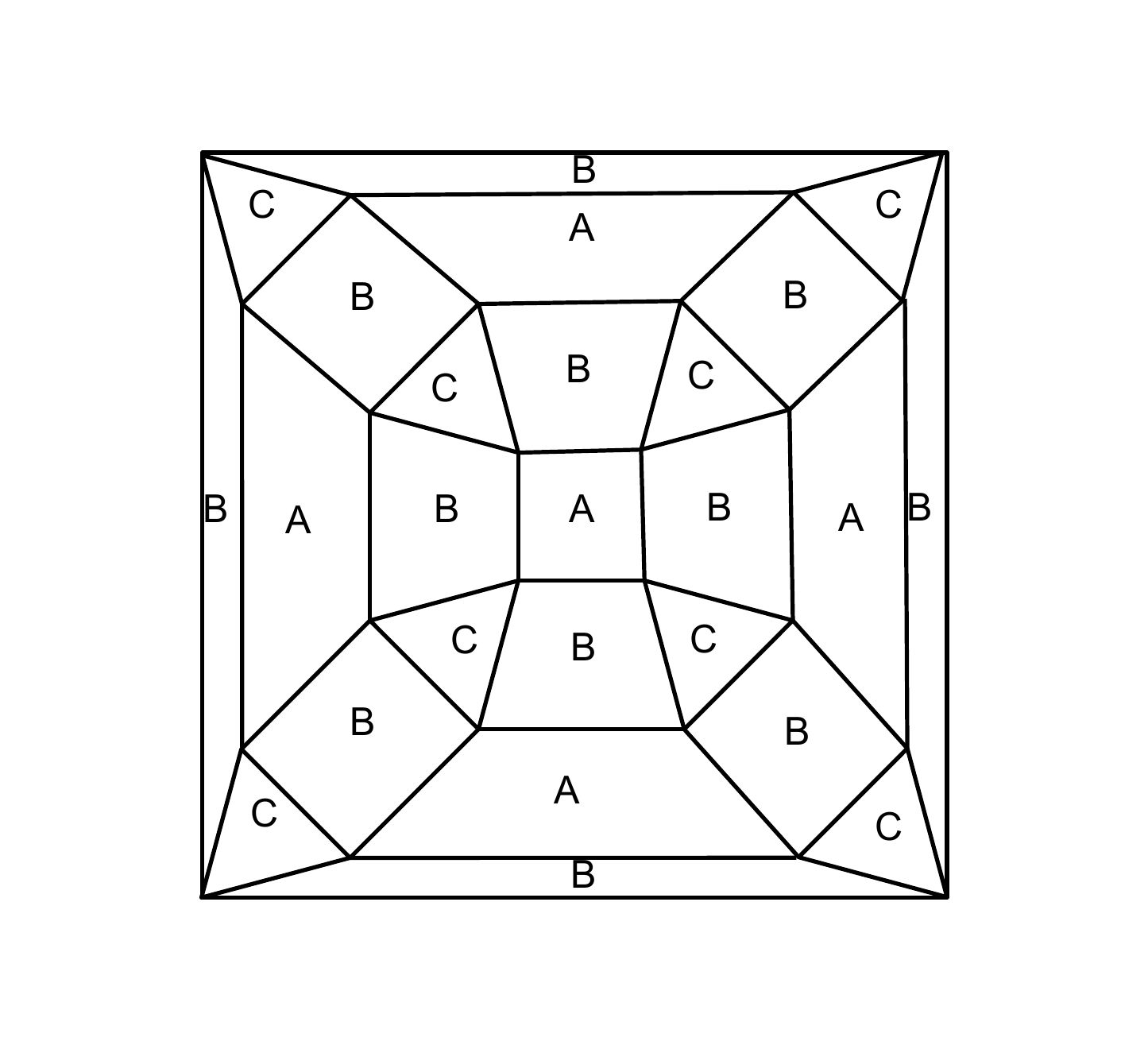}} \caption{The $R$-complex $X$ that we start with for the subdivision rule $R$ in Figure \ref{TorusSubdivision}. This is a Schlegel diagram of a polyhedron; the outside face is an $A$ face.}
\label{TorusS0}
\end{center}
\end{figure}

\begin{figure}
\begin{center}
\scalebox{.4}{\includegraphics{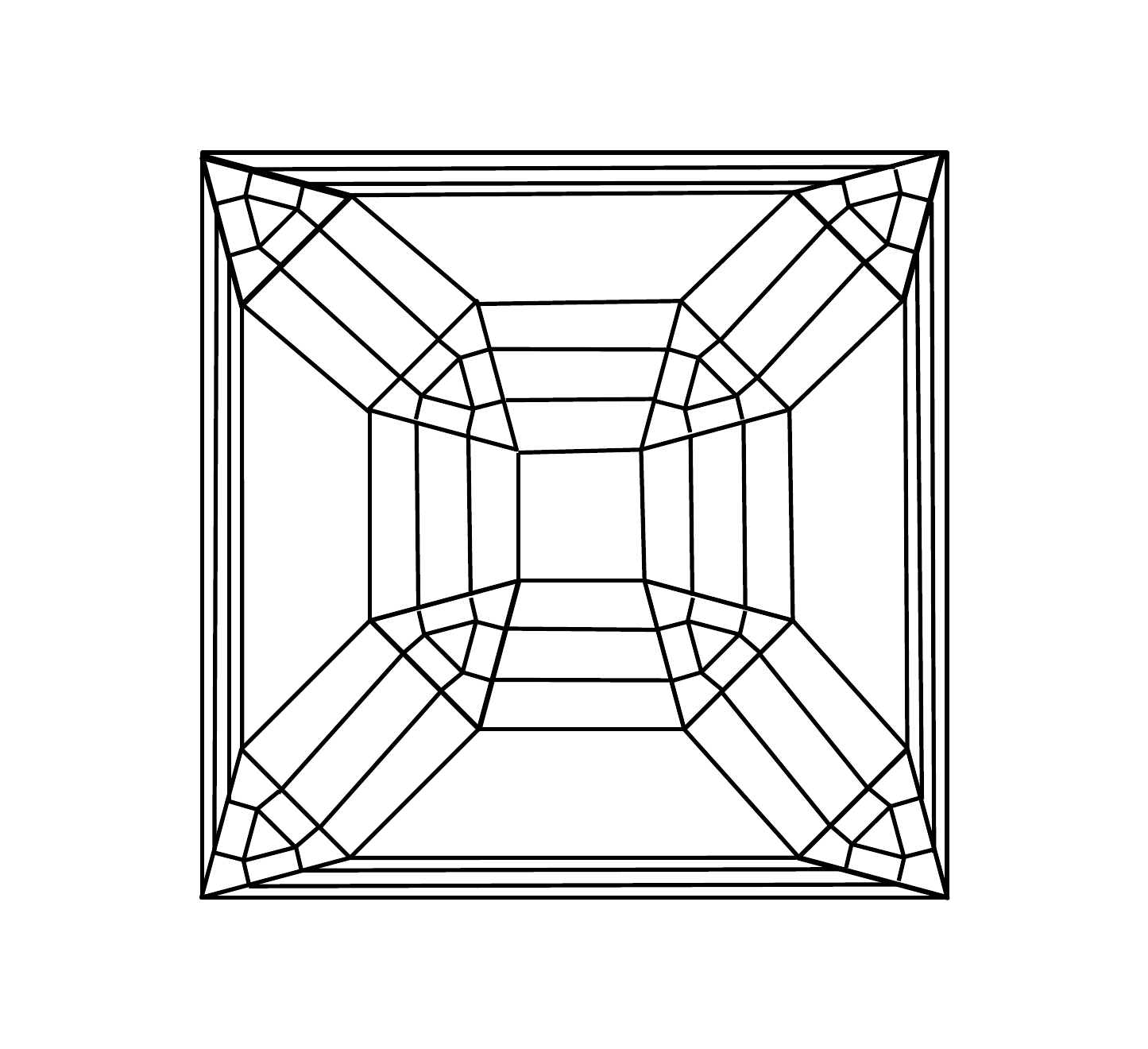}} \caption{The first subdivision $R(X)$ of the $R$-complex $X$ in Figure \ref{TorusS0}.}
\label{TorusS1}
\end{center}
\end{figure}

We can use Theorem \ref{Quasi1Thm} to characterize those subdivision rules and complexes whose history graphs are quasi-isometric to a Euclidean space of dimension 2 or 3.

We recall several preliminary definitions.

\begin{definition}
Two groups $G_1,G_2$ are \textbf{commensurable} if they contain finite index subgroups $H_1\subseteq G_1$, $H_2\subseteq G_2$ such that $H_1$ and $H_2$ are isomorphic. It is easy to show that this is an equivalence relation. The equivalence classes under this relation are called \textbf{commensurability classes}.
\end{definition}

If two groups are commensurable, then their history graphs are quasi-isometric (see Section 1 of \cite{farb2000problems}).

\begin{definition}
A finitely generated group is said to be \textbf{virtually nilpotent} if it contains a finite-index subgroup which is nilpotent. The term \textbf{nilpotent-by-free} is also used for virtually nilpotent groups (see, for instance, the introduction to \cite{alperin1999solvable}).
\end{definition}

Gromov's theorem on groups of polynomial growth \cite{gromov1981groups} says that every group of polynomial growth is virtually nilpotent.

For groups with growth functions of degree 2 or 3, Gromov's theorem can be further refined (Proposition 4.8a of \cite{mann2012groups}):

\begin{theorem}\label{MannTheorem}
Within the class of nilpotent-by-finite groups we have that all groups of quadratic growth lie in the same commensurability class, and all groups of cubic growth lie in the same commensurability class.
\end{theorem}

Combining these theorems with Theorem \ref{Quasi1Thm}, we have the following:

\begin{theorem}\label{EuclideanTheorem}
Let $R$ be a subdivision rule, and let $X$ be a finite $R$-complex. Assume that $\Gamma(R,X)$ is quasi-isometric to a Cayley graph of a group $G$. Then:
\begin{enumerate}
\item the counting function $c_X$ is a quadratic polynomial if and only if $\Gamma(R,X)$ is quasi-isometric to $\mathbb{E}^2$, the Euclidean plane, and
\item the counting function $c_X$ is a cubic polynomial if and only if $\Gamma(R,X)$ is quasi-isometric to $\mathbb{E}^3$, Euclidean space.
\end{enumerate}
\end{theorem}
\begin{proof}
Assume that $\Gamma(R,X)$ is quasi-isometric to a Cayley graph of a group $G$, and that $c_X$ is a quadratic polynomial. By Theorem \ref{Quasi1Thm}, the group $G$ has quadratic growth. Then Gromov's theorem on groups of polynomial growth implies that $G$ is virtually nilpotent. Furthermore, Theorem \ref{MannTheorem} implies that $G$ is commensurable to $\mathbb{Z}^2$, which also has quadratic growth. Thus, every Cayley graph of $G$ is quasi-isometric to $\mathbb{Z}^2$, and $\Gamma(R,X)$ is quasi-isometric to $\mathbb{E}^2$. The converse holds by Theorem \ref{Quasi1Thm}.

The proof is essentially the same in the cubic case.
\end{proof}

An example of a subdivision rule $R$ and an $R$-complex $X$ whose history graph is quasi-isometric to 3-dimensional Euclidean space is shown in Figures \ref{TorusSubdivision}-\ref{TorusS1}.

\subsection{Geometries without subdivision rules: Nil and Sol}

In Section \ref{Combable}, we showed that all history graphs are combable, and therefore any groups quasi-isometric to them satisfy a quadratic isoperimetric inequality. It is known that a group with Nil geometry has a cubic isoperimetric function and that a group with Sol geometry has an exponential isoperimetric function (Example 8.1.1 and Theorem 8.1.3, respectively, of \cite{epstein1992word}). Thus, we have the following:

\begin{cor}\label{NilSol}
A history graph of a finite subdivision rule cannot be quasi-isometric to Nil geometry or Sol geometry.
\end{cor}

\subsection{$\mathbb{H}^2\times \mathbb{R}$ and $\widetilde{SL_2}(\mathbb{R})$ geometries}
\label{H2RSection}

\begin{figure}
\begin{center}
\scalebox{.4}{\includegraphics{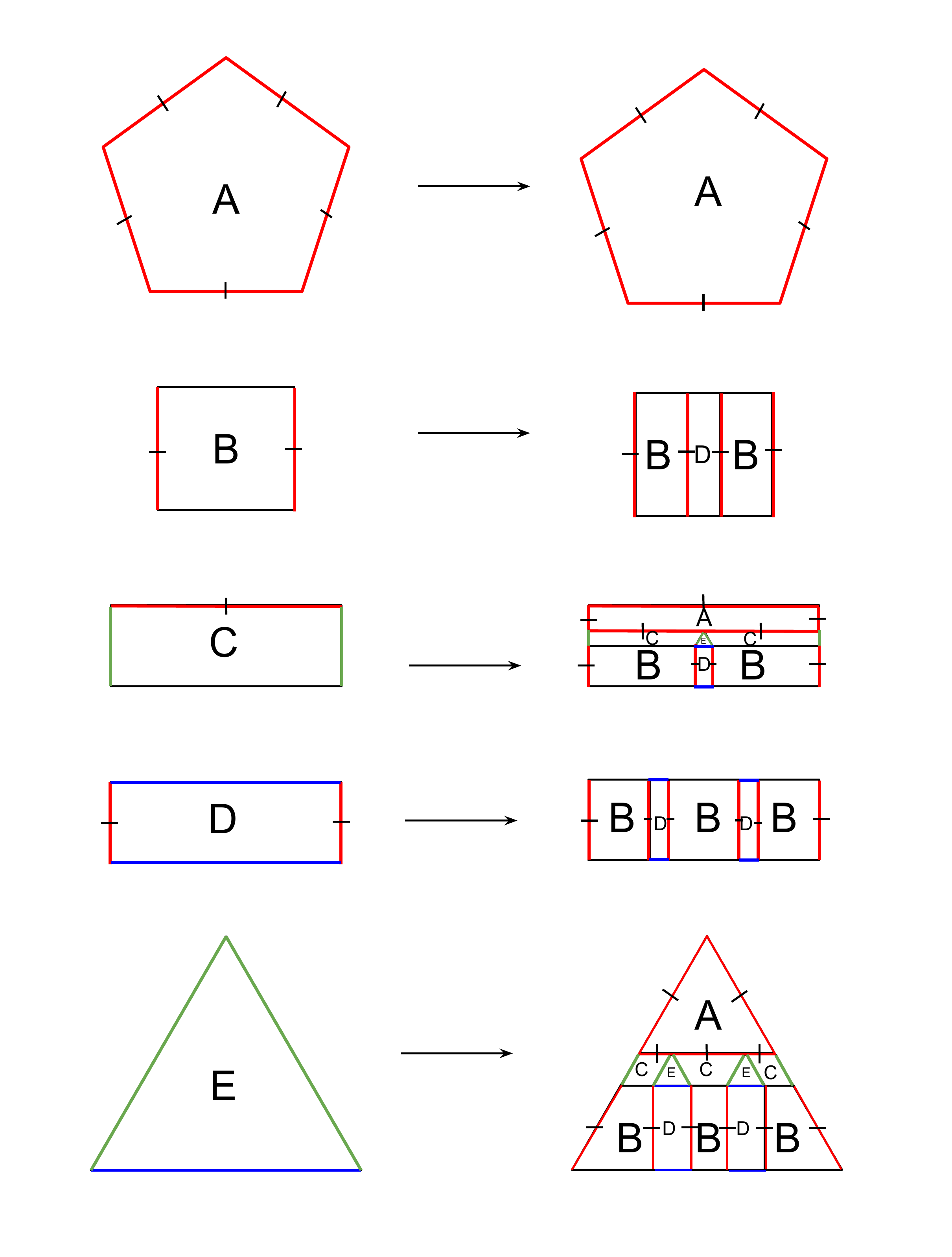}} \caption{A subdivision rule $Q$. With the $Q$-complex shown in Figure \ref{ProductSubsComplex}, we obtain a history graph with $\mathbb{H}^2\times \mathbb{R}$ geometry. The colors indicate various edge types. All green edges are oriented from top to bottom; all blue and black edges are oriented from left to right. Because the red edges are oriented in many different directions, we place a vertex (indicated by a black line) in the midpoint of each red edge, and orient each of the two new edges towards this midpoint. This is similar to the `barycenter trick' described on p.12 of \cite{cannon2001finite}.}
\label{ByHandProductSubs}
\end{center}
\end{figure}

\begin{figure}
\begin{center}
\scalebox{.2}{\includegraphics{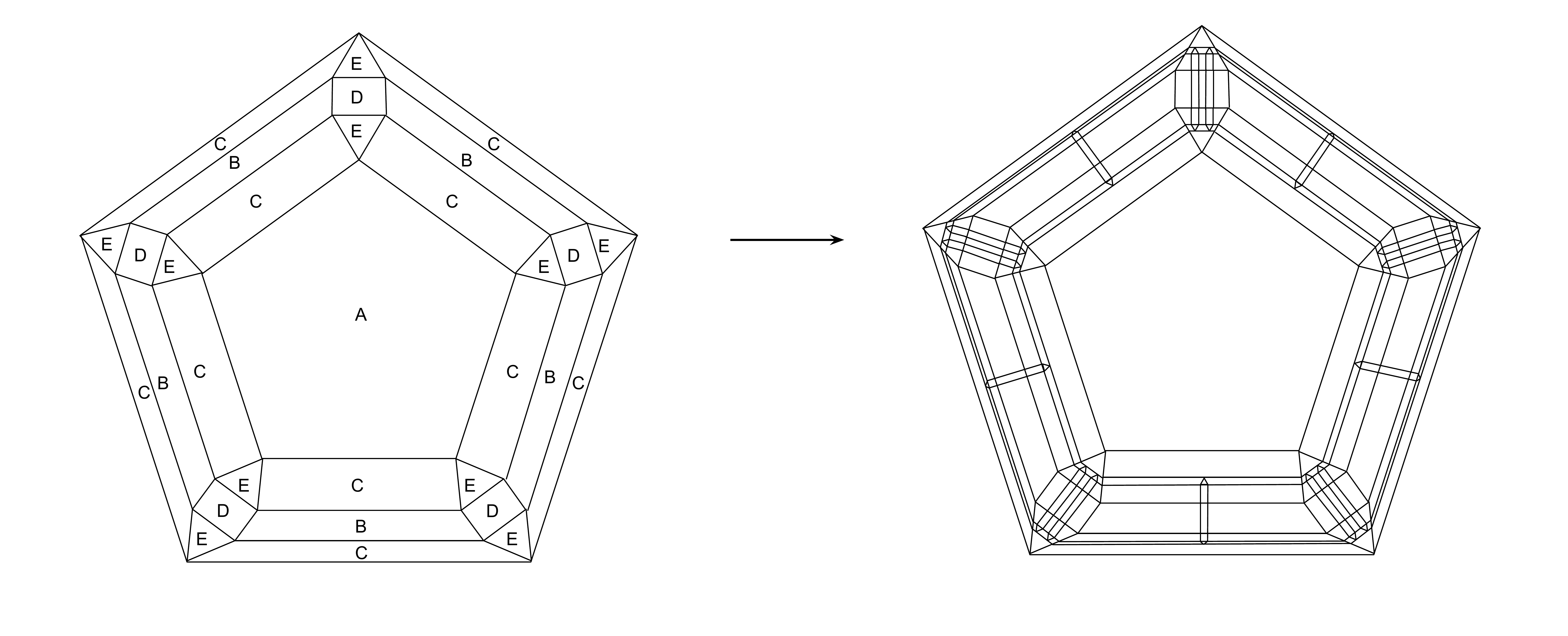}} \caption{A $Q$-complex for the subdivision rule $Q$ in Figure \ref{ByHandProductSubs}, together with its first subdivision. These are Schlegel diagrams of polyhedra; the outside face is an $A$ face. We omit the colors on the edges and the midpoints used in orienting the edges in Figure \ref{ByHandProductSubs}.}
\label{ProductSubsComplex}
\end{center}
\end{figure}

In this section, we study the product geometry $\mathbb{H}^2 \times \mathbb{R}$ and its sister geometry $\widetilde{SL_2}(\mathbb{R})$, which are quasi-isometric. We do not have a general  characterization for this pair of geometries. However, we can distinguish them from the other geometries.

\begin{theorem}
If $\Gamma(R,X)$ is quasi-isometric to $\mathbb{H}^2 \times \mathbb{R}$, then:
\begin{enumerate}
\item\label{growthh2r} $\Gamma(R,X)$ has a growth function that grows exponentially.
\item \label{hyperh2r} $R$ is not hyperbolic with respect to $X$.
\end{enumerate}
\end{theorem}
\begin{proof}
Property \ref{growthh2r} holds by Theorem \ref{Quasi1Thm}. Property \ref{hyperh2r} holds by Theorem \ref{HyperHausdorff}.
\end{proof}

\begin{theorem}
If $\Gamma(R,X)$ is quasi-isometric to a model geometry of dimension less than 4, then it is quasi-isometric to $\mathbb{H}^2\times \mathbb{R}$ if it has exponential growth and is not hyperbolic.
\end{theorem}

\begin{proof}
The growth being exponential rules out all geometries except $\mathbb{H}^2$, $\mathbb{H}^3$, $\mathbb{H}^2\times \mathbb{R}$, and Sol. We know that Sol geometries cannot be modeled by subdivision rules by Corollary \ref{NilSol}, and we know that $\mathbb{H}^n$ is hyperbolic. Thus, the only remaining geometry is $\mathbb{H}^2\times \mathbb{R}$.
\end{proof}

In Figures \ref{ByHandProductSubs} and \ref{ProductSubsComplex}, we show a subdivision rule $Q$ and $Q$-complex $X$ whose history graph $\Gamma(Q,X)$ is quasi-isometric to $\mathbb{H}^2 \times \mathbb{R}$.
\section{Future Work}

We hope to find more explicit characterizations of the product geometries, as well as characterizing subdivision rules for relatively hyperbolic groups.

\bibliographystyle{amsplain}
\bibliography{Asympbib}
\end{document}